\documentclass[12pt]{article}
\usepackage{amssymb,amsmath}
\usepackage{amsthm}

\newcommand{\ul}{\underline}

\newcommand{\e}{\mathrm{e}}

\newcommand{\erf}{\mathop{\mathrm{erf}}\nolimits}

\newtheorem{theorem}{Theorem}[section]
\newtheorem{corollary}[theorem]{Corollary}
\newtheorem{lemma}[theorem]{Lemma}

\theoremstyle{definition}
\newtheorem{definition}[theorem]{Definition}
\newtheorem{example}[theorem]{Example}
\theoremstyle{remark}
\newtheorem*{term}{Terms in this paper}
\newtheorem*{remark}{Remark}

\newtheorem*{acknowledgment}{Acknowledgments}

\title{Difference-differential fields of continuous functions 
}

\author{Seiji NISHIOKA\footnote{
  Faculty of Science, Yamagata University, 1-4-12 Kojirakawa-machi,  Yamagata-shi, Yamagata, 990-8560, Japan. e-mail: nishioka@sci.kj.yamagata-u.ac.jp
  }
}

\date{\today}

\begin{document}

\maketitle

\begin{abstract}
  \noindent
  The set $\mathcal{C}$ of complex-valued continuous functions on $[0,\infty)$
  is a ring by the addition and the convolution.  
  It has the quotient field $Q(\mathcal{C})$, 
  by which J. Mikusi\'nski developed his operational calculus.  
  In this paper, 
  we revisit a derivation and a transforming operator for $Q(\mathcal{C})$
  written in his textbook,
  and define another transforming operator  
  related to the $q$-shift operator, which
  gives structures of a $q$-difference field
  and a difference field of Mahler type to $Q(\mathcal{C})$.  
  Appropriate derivatives are also considered.  
  

  
\end{abstract}

\section{Introduction}
Besides the field of meromorphic functions and the field of formal power series, 
we find another example of differential/difference field in J. Mikusi\'nski's 
study \cite{Mikusinski1983}.  
His operational calculus is founded on the commutative ring $\mathcal{C}$ of
complex-valued continuous functions on $[0,\infty)$
by the addition and the convolution, 
where the convolution of $f(t)$ and $g(t)$ is 
\[
  \int_0^tf(t-\tau)g(\tau)d\tau.
\]
The ring $\mathcal{C}$ has its quotient field $Q(\mathcal{C})$ by Titchmarsh's theorem.  
In his textbook \cite{Mikusinski1983}, a continuous function $f\in\mathcal{C}$ is denoted by the symbol $\{f(t)\}$. 
The identity $1$ is not the constant function $l=\{1\}$ but $\{1\}/\{1\}$.  
Any complex number $\alpha\in\mathbb{C}$ is also regarded as $\{\alpha\}/\{1\}$.  
The function $l$ is called the integral operator, 
for its convolution with an arbitrary function $f=\{f(t)\}$ causes
\[
  lf=\{1\}\{f(t)\}=\left\{\int_0^tf(\tau)d\tau\right\}.
\]
Its inverse $1/l$ in the field $Q(\mathcal{C})$ is denoted by the letter $s$ and called the differential operator.  
It is known that the exponential function $\e^{\alpha t}$ 
and Bessel function $J_0(\alpha t)$ are
expressed as 
\[
  \{\e^{\alpha t}\}=\frac{1}{s-\alpha},\quad
  \{J_0(\alpha t)\}=\frac{1}{\sqrt{s^2+\alpha^2}}.
\]  
In addition to the above two arithmetical operators, the field $Q(\mathcal{C})$
has other algebraic operations.  
After Mikusi\'nski, the derivation $d/ds$ is defined by 
\[
  \frac{d}{ds}\{a(t)\}=\{-ta(t)\}.
\]
As the symbol suggests, the derivative of $s$ is $1$.  
He also defined the transforming operator $T^\alpha$ by
\[
  T^\alpha\{a(t)\}=\{\e^{\alpha t}a(t)\}
\]
for an arbitrary complex number $\alpha$.  
It is an automorphism of $Q(\mathcal{C})$ and has the following formula, 
\[
  T^\alpha s=s-\alpha,
\]
which provides a usual structure of a difference field for $\mathbb{C}(s)$.  

All of the above is found in Mikusi\'nski's textbook \cite{Mikusinski1983}.  
In this paper, we define another transforming operator $\tau_q$ by
\[
  \tau_q\{a(t)\}=\{qa(qt)\}.
\]
Since it satisfies $\tau_ql=ql$, $\mathbb{C}(l)$ is a $q$-difference field with it.  
Moreover, a difference field of Mahler type, which is related to transcendental number theory (cf. Ku. Nishioka's book \cite{Kumi:book}), appears when $q$ is a unit fraction 
(see Theorem \ref{p16a}).  

Once we find $Q(C)$ is a differential/difference field, 
purely algebraic results in differential/difference algebra are applicable to it.  
For example, setting $L=Q(C)$, 
we can use the following theorem to prove algebraic independence of 
solutions of first-order rational difference equations.  

\begin{theorem}[S. Nishioka \cite{S2013}, Theorem 1]\label{th_deg}
  Let $\mathcal{K}$ be a difference field of characteristic zero and 
  $\mathcal{L}=(L,\tau)$ a difference overfield of $\mathcal{K}$.  
  Let $c_1,\dots,c_n\in\mathbb{Z}_{>0}$ $(n\geq 1)$ with $c_i\neq c_j$ $(i\neq j)$
  and let 
  \begin{equation*}
    f_{11},\dots,f_{1m_1},\ 
    f_{21},\dots,f_{2m_2},\ 
    \dots,\ 
    f_{n1},\dots,f_{nm_n}\in L \quad (m_i\geq 1)
  \end{equation*}
  satisfy $$\tau(f_{ij})B_{ij}(f_{ij})=A_{ij}(f_{ij}),$$ where
  $A_{ij},B_{ij}\in K[X]\setminus\{0\}$ are relatively prime
  and satisfy $$\max\{\deg A_{ij},\deg B_{ij}\}=c_i.$$
  If $f_{i1},\dots,f_{im_i}$ are 
  algebraically independent over $K$ for each $1\leq i\leq n$,
  then the elements $f_{ij}$, $1\leq i\leq n$, $1\leq j\leq m_i$,
  are algebraically independent over $K$.
\end{theorem}
Indeed, we will see algebraic independence of the differential operator $s$ and the translation operator $h^\lambda$ ($\lambda>0$, see Definition \ref{p5a}) 
over $\mathbb{C}$ in Section \ref{s:dcfield}.  
This implies that the translation operator $h^\lambda$ cannot be written
  algebraically by $s$ unlike the exponential function or Bessel function.

\begin{term}
  A derivation $D\colon R\to R$ of a ring $R$ is a mapping which satisfies
  \[
    D(a+b)=D(a)+D(b),\quad D(ab)=D(a)b+aD(b).
  \]
  A transforming operator $\tau\colon R\to R$ is an injective endomorphism,
  or isomorphism of $R$ into itself which may not be surjective. 
  The automorphism $\tau\colon f(x)\mapsto f(x+1)$ of the field of meromorphic functions
  is a typical example (cf. introductory books \cite{Cohn,Levin2008,PS:book}).  
  A ring with a derivation is called a differential ring,
  and one with a transforming operator a difference ring.
  In this paper, we call a ring with both of them a DT ring, 
  which is short for a ring with a \ul{d}erivation and 
  a \ul{t}ransforming operator.  
\end{term}

In Section \ref{s:opcalc}, we introduce some parts of  
Mikusi\'nski's operational calculus.  
Section \ref{s:dcfield} begins with a definition of the transforming operator $\tau_q$
and includes the way $Q(\mathcal{C})$ has structures
of a $q$-difference field and also a difference field of Mahler type.  
Section \ref{s:qdc} and \ref{s:Mahler} contain
DT rings of convergent power series $\mathbb{C}\{l\}$
and a DT ring of formal power series $\mathbb{C}[[h]]$ respectively.  

\section{Mikusi\'nski's operational calculus}\label{s:opcalc}

In this section, we introduce some parts of 
Mikusi\'nski's operational calculus.  
For the omitted proofs, 
see Mikusi\'nski's textbook \cite{Mikusinski1983}.

Let $\mathcal{C}$ be the set of complex-valued continuous functions 
on $[0,\infty)$.  
A continuous function $f\in\mathcal{C}$ is denoted by the symbol $\{f(t)\}$.  
The set $\mathcal{C}$ is a commutative ring without identity by 
\begin{gather*}
  f+g=\{f(t)\}+\{g(t)\}=\{f(t)+g(t)\},\\
  fg=\{f(t)\}\{g(t)\}=\left\{\int_0^tf(t-\tau)g(\tau)d\tau\right\}.
\end{gather*}
The constant function $l=\{1\}$ is called the integral operator, for
the multiplication of $l$ and $f\in\mathcal{C}$ causes
\[
  lf=\{1\}\{f(t)\}=\left\{\int_0^t 1\cdot f(\tau)d\tau\right\}
  =\left\{\int_0^tf(\tau)d\tau\right\}.
\]
By Titchmarsh's theorem, $\mathcal{C}$ has its quotient field $Q(\mathcal{C})$.

\begin{theorem}[Titchmarsh's theorem]
  For any $f,g\in\mathcal{C}$, $fg=0$ implies $f=0$ or $g=0$.  
\end{theorem}

For any constant function $\{\alpha\}$, $\{\alpha\}/\{1\}\in Q(\mathcal{C})$
is simply denoted by $\alpha$.  
The identity of the field $Q(\mathcal{C})$ is $1=\{1\}/\{1\}$, and by 
\[
  \{1\}\{0\}=\left\{\int_0^t 0\ d\tau\right\}
  =\{0\},
\]
the zero is $0=\{0\}$.
Then $\mathbb{C}$ is regarded as a subfield of $Q(\mathcal{C})$, for
\begin{gather*}
  \frac{\{\alpha\}}{\{1\}}+\frac{\{\beta\}}{\{1\}}
  =\frac{\{\alpha\}+\{\beta\}}{\{1\}}
  =\frac{\{\alpha+\beta\}}{\{1\}},\\
  \frac{\{\alpha\}}{\{1\}}\cdot\frac{\{\beta\}}{\{1\}}
  =\frac{\left\{\int_0^t\alpha\beta d\tau\right\}}{\{1\}^2}
  =\frac{l\{\alpha\beta\}}{l\{1\}}
  =\frac{\{\alpha\beta\}}{\{1\}}.
\end{gather*}
In addition, for $\alpha\in\mathbb{C}$, 
\[
  \alpha\{f(t)\}=\frac{\{\alpha\}}{\{1\}}\{f(t)\}
  =\frac{\left\{\int_0^t \alpha f(\tau)d\tau\right\}}{\{1\}}
  =\{\alpha f(t)\}.
\]
We call $s=1/l$ the differential operator. Indeed, 
if $f\in\mathcal{C}$ is a $C^1$-function, multiply
\[
  lf'=l\{f'(t)\}=\left\{\int_0^t f'(\tau)d\tau\right\}
  =\{f(t)-f(0)\}=f-\{f(0)\}
\]
by $s=1/\{1\}$. Then we obtain $f'=sf-f(0)$. 

\begin{example}\label{p2a}
  The following are elements of the field $\mathbb{C}(l)=\mathbb{C}(s)$, 
  \begin{align*}
    l^n=&\left\{\frac{t^{n-1}}{(n-1)!}\right\}\quad (n\in\mathbb{Z}_{>0}),\\
    \frac{1}{s-\alpha}=&\{\e^{\alpha t}\},\\
    \frac{1}{(s-\alpha)^n}=&\left\{\frac{t^{n-1}}{(n-1)!}\e^{\alpha t}\right\}
    \quad (n\in\mathbb{Z}_{>0}),\\
    \frac{1}{(s-\alpha)^2+\beta^2}=&\left\{\frac{1}{\beta}\e^{\alpha t}\sin \beta t\right\}\quad (\beta\neq 0),\\
    \frac{s-\alpha}{(s-\alpha)^2+\beta^2}=&\{\e^{\alpha t}\cos \beta t\}.
  \end{align*}
  For example, the second equality is seen by
  \[
    s\{\e^{\alpha t}\}=\{(\e^{\alpha t})'\}+\e^{\alpha 0}=\alpha\{\e^{\alpha t}\}+1.  
  \]
  By the first equation, we find $l\notin\mathbb{C}$.  
  Indeed, while $\alpha\in\mathbb{C}$ multiplied by $l$ is a constant function $\{\alpha\}$, $l$ multiplied by itself is $l^2=\{t\}$.  
  Hence $l$ and $s$ are transcendental over $\mathbb{C}$.  
  
\end{example}

\begin{definition}
  If a complex-valued function $f$ on $[0,\infty)$ has 
  a continuous primitive function 
  $\int_0^t f(\tau)d\tau$ $(0\leq t<\infty)$, then
  \[
    \frac{\{\int_0^t f(\tau)d\tau\}}{l}\in Q(\mathcal{C}).
  \]
  We shall denote it by the symbol $\{f(t)\}^\sharp$. 
  When $f\in \mathcal{C}$, we see
  \[
    \{f(t)\}^\sharp=\frac{\{\int_0^t f(\tau)d\tau\}}{l}
    =\frac{lf}{l}=f=\{f(t)\}.
  \]
  Note that $\{f(t)\}^\sharp=\{g(t)\}^\sharp$ does not mean $f(t)=g(t)$ for every point.  
  Actually, Mikusi\'nski's treatment to discontinuous functions is slightly different from ours (cf. \cite{Mikusinski1983}, Part I, Chapter VII, \S56).  
  While he redefined the equality $\{f(t)\}=\{g(t)\}$ by ignoring a few points, 
  we do not change the meaning of the equality (cf. Fujimoto \cite{Fujimoto}).  
\end{definition}

\begin{definition}[Part I, Chapter VII, \S59]\label{p3b}
  For any $\alpha\in\mathbb{C}$ and $\lambda\in\mathbb{R}$, we define
  a power of $s-\alpha$ by
  \[
    (s-\alpha)^{-\lambda}
    =\begin{cases}
      \left\{\dfrac{t^{\lambda-1}}{\Gamma(\lambda)}\e^{\alpha t}\right\}^\sharp & (\lambda>0),\\
      1 & (\lambda=0),\\
      \dfrac{1}{(s-\alpha)^\lambda} & (\lambda < 0).
    \end{cases}
  \]
  When $\alpha=0$, $s^{-\lambda}$ is denoted by $l^\lambda$.  
  They satisfy
  \[
    (s-\alpha)^{-\lambda}(s-\alpha)^{-\mu}=(s-\alpha)^{-\lambda-\mu}.
  \]
\end{definition}

\begin{example}[Part I, Chapter VII, \S59]\label{p4a}
  By the above formula, $\sqrt{s+\alpha}=(s+\alpha)^{1/2}\in Q(\mathcal{C})$
  satisfies
  $(\sqrt{s+\alpha})^2=s+\alpha$,
  and thus is algebraic over $\mathbb{C}(s)$.  
  When $\alpha>0$, it is related to the error function in the following way,
  \[
    \frac{1}{s\sqrt{s+\alpha}}
    =\left\{\frac{1}{\sqrt{\alpha}}\erf(\sqrt{\alpha t})\right\},\quad
    \erf(t)=\frac{2}{\sqrt{\pi}}\int_0^t\e^{-\tau^2}d\tau.  
  \]
\end{example}

\begin{definition}[Part I, Chapter VII, \S62]\label{p5a}
  Let $\lambda\geq 0$. we define the jump function, or Heaviside's function,
  by
  \[
    H_\lambda(t)=\begin{cases}
      0 & (0\leq t\leq \lambda),\\
      1 & (\lambda <t<\infty),
    \end{cases}
  \]
  and the translation operator by $h^\lambda=s\{H_\lambda(t)\}^\sharp$.  
  Indeed, for any $\lambda>0$, $f\in\mathcal{C}$ is shifted to
  \[
    h^\lambda f=\{g_\lambda(t)\}^\sharp,\quad
    g_\lambda(t)=\begin{cases}
      0 & (0\leq t\leq \lambda),\\
      f(t-\lambda) & (\lambda < t< \infty).  
    \end{cases}
  \]
  Hence for any $\lambda>0$ and $\mu>0$, we see
  \begin{equation}\label{p5c}
    h^\lambda h^\mu=h^{\lambda+\mu}.
  \end{equation}
  Since 
  \[
    h^0=s\frac{\left\{\int_0^t H_0(\tau)d\tau\right\}}{l}
    =\frac{s}{l}\{t\}=\frac{s}{l}l^2=1,
  \]
  this formula holds for any $\lambda\geq 0$ and $\mu\geq 0$. 
  Furthermore, 
  if we define negative-powers of $h$ by
  $h^{-\lambda}=1/h^\lambda$ $(\lambda>0)$,
  the formula \eqref{p5c} holds for any $\lambda,\mu\in\mathbb{R}$.
  In addition, if we simply write $h$ instead of $h^1$, 
  we see 
  \[
    (h)^n=(h^1)^n=h^n
  \]
  for any $n\in\mathbb{Z}$.  
  When $\lambda>0$, we find $h^\lambda\notin\mathbb{C}$.  
  Indeed, while $\alpha\in\mathbb{C}$ multiplied by $l^2$ is 
  $\alpha l^2=\alpha\{t\}=\{\alpha t\}$, 
  \[
    h^\lambda l^2=h^\lambda\{t\}=\{g_\lambda(t)\},\quad
    g_\lambda(t)=\begin{cases}
      0 & (0\leq t\leq \lambda),\\
      t-\lambda & (\lambda < t< \infty). 
    \end{cases}
  \]
  Hence for $\lambda\neq 0$, $h^\lambda$ is transcendental over $\mathbb{C}$.  
  
\end{definition}

\begin{definition}[Part II, Chapter I, \S2--3]\label{p6a}
  We say that a sequence $\{a_n\}_n$ in $Q(\mathcal{C})$ converges 
  if there exists a certain $p\in Q(\mathcal{C})$ such that 
  $\{a_n/p\}_n$ is a sequence in $\mathcal{C}$ and converges uniformly on
  every finite interval.  
  Since the functions $a_n/p$ are continuous, the limit is also continuous.  
  Therefore the limit of $\{a_n\}_n$ is defined by
  \[
    \lim_{n\to\infty}a_n=p\lim_{n\to\infty}\frac{a_n}{p}\in Q(\mathcal{C}).
  \]
  We find that the limit is unique, and
  if $\{a_n\}_n,\{b_n\}_n\subset Q(\mathcal{C})$ converges to $a,b\in Q(\mathcal{C})$ respectively, 
  \[
    \lim_{n\to\infty}(a_n\pm b_n)=a\pm b,\quad
    \lim_{n\to\infty}a_nb_n=ab.
  \]
  Under additional conditions that $b_n\neq 0$ $(n=0,1,2,\dots)$, $b\neq 0$
  and $\{a_n/b_n\}_n$ converges, we see
  \[
    \lim_{n\to\infty}\frac{a_n}{b_n}=\frac{a}{b}. 
  \]
  Series $\sum_{n=0}^\infty a_n$ ($a_n\in Q(\mathcal{C})$)
  is defined in the usual way.  
\end{definition}

\begin{theorem}[Part II, Chapter II, \S4]\label{p7a}
  A series 
  \[
    \sum_{n=0}^\infty\alpha_n h^{\beta_n},\quad
    \alpha_n\in\mathbb{C},\ \beta_n\in\mathbb{R}_{\geq 0}.
  \]
  converges to  
  \[
    \frac{1}{l^2}\left\{\sum_{n=0}^\infty\alpha_n\max\{0,t-\beta_n\}\right\}
  \]
  if $\{\beta_n\}_n$ is monotonically increasing and $\lim_{n\to\infty}\beta_n=+\infty$. 
\end{theorem}

\begin{proof}
  Since $h^{\beta^n}$ multiplied by $l^2$ is 
  \[
    l^2h^{\beta_n}=h^{\beta_n}\{t\}
    =\left\{\begin{aligned}
      &0&&(0\leq t\leq\beta_n),\\&t-\beta_n&&(\beta_n<t<\infty)
    \end{aligned}\right\}
    =\{\max\{0,t-\beta_n\}\},
  \]
  the partial sum multiplied by the same one is 
  \[
    l^2\sum_{n=0}^N\alpha_nh^{\beta_n}
    =\sum_{n=0}^N\alpha_n\{\max\{0,t-\beta_n\}\}
    =\left\{\sum_{n=0}^N\alpha_n\max\{0,t-\beta_n\}\right\}, 
  \]
  and thus a continuous function.  
  If $N\geq N_0$, then for every $0\leq t\leq \beta_{N_0}$,
  \[
    \sum_{n=0}^N\alpha_n\max\{0,t-\beta_n\}
    =\sum_{n=0}^{N_0}\alpha_n\max\{0,t-\beta_n\}.  
  \]
  Hence $\sum_{n=0}^N\alpha_n\max\{0,t-\beta_n\}$ converges uniformly
  on every finite interval, and thus by definition, 
  \[
    \sum_{n=0}^\infty\alpha_nh^{\beta_n}
    =\frac{1}{l^2}\left\{\sum_{n=0}^\infty\alpha_n\max\{0,t-\beta_n\}\right\}.
  \]
\end{proof}

\begin{lemma}[Part II, Chapter IV, \S15]\label{p8a}
  Let $\mathbb{C}\{X\}$ denote the ring of convergent power series and 
  $\sum_{n=1}^\infty \alpha_nX^n\in\mathbb{C}\{X\}$.  
  Let $f\in\mathcal{C}$, $\lambda_0>0$ and $t_0>0$.  
  Then $\sum_{n=1}^\infty\alpha_n\lambda^nf^n$ converges uniformly 
  on the set
  \[
    \{(\lambda,t)\ |\ 0\leq\lambda\leq \lambda_0,\, 
    0\leq t\leq t_0\}
  \]
  as a series of functions of two variables, 
  $\lambda$ and $t$, 
  where $f^n$ denotes a power by convolution.  
\end{lemma}

\begin{theorem}[Part II, Chapter IV, \S15]\label{p8b}
  Let 
  $\sum_{n=0}^\infty\alpha_nX^n\in\mathbb{C}\{X\}$.  
  Then for any $f\in\mathcal{C}$, 
  $\sum_{n=1}^\infty \alpha_nf^n$ converges uniformly on every finite interval 
  in $[0,\infty)$ as a series of functions of one variable $t$.  
  By its limit $g\in\mathcal{C}$, 
  $\sum_{n=0}^\infty\alpha_nf^n=\alpha_0+g\in Q(\mathcal{C})$.  
\end{theorem}

\begin{proof}
  Let $t_0>0$. By the above lemma, $\sum_{n=1}^\infty\alpha_n\lambda^nf^n$
  converges uniformly on the set
  \[
    \{(\lambda,t)\ |\ 0\leq\lambda\leq 1,\, 
    0\leq t\leq t_0\}.
  \]
  Hence $\sum_{n=1}^\infty\alpha_nf^n$ converges uniformly on $0\leq t\leq t_0$,
  and thus on every finite interval in $[0,\infty)$. 
\end{proof}

\begin{example}\label{p8c}
  Let $\sum_{n=1}^\infty\alpha_nX^n\in\mathbb{C}\{X\}$.  
  Setting $f=l$ in the theorem, we obtain
  \[
    \sum_{n=1}^\infty\alpha_nl^n
    =\sum_{n=1}^\infty\alpha_n\left\{\frac{t^{n-1}}{(n-1)!}\right\}
    =\left\{\sum_{n=1}^\infty\frac{\alpha_nt^{n-1}}{(n-1)!}\right\}.
  \]
\end{example}

\begin{example}[Part II, Chapter IV, \S18]\label{p10a}
  Let 
  \[
    J_0(t)=\sum_{k=0}^\infty\frac{(-1)^k}{(k!)^2}\left(\frac{t}{2}\right)^{2k}
  \]
  be the Bessel function and $\alpha\in\mathbb{C}$.
  Then we see
  \[
    \{J_0(\alpha t)\}
    =\sum_{k=0}^\infty\binom{-\frac{1}{2}}{k}\alpha^{2k}l^{2k+1}.
  \]
  By the Taylor expansion of 
  \[
    \frac{1}{\sqrt{1+\lambda}}
    =\sum_{k=0}^\infty\binom{-\frac{1}{2}}{k}\lambda^k
    \quad (|\lambda|<1),
  \]
  it is usually denoted by
  \[
    \{J_0(\alpha t)\}=\frac{l}{\sqrt{1+(\alpha l)^2}}
    =\frac{1}{\sqrt{s^2+\alpha^2}}.  
  \]
\end{example}

\begin{theorem}[Part II, Chapter IV, \S14]\label{p9a}
  Let $\lambda_1>0$.  If
  \[
    \sum_{n=0}^\infty a_n\lambda^n,\ 
    \sum_{n=0}^\infty b_n\lambda^n,\quad
    a_n,b_n\in Q(\mathcal{C})
  \]
  converge for every $|\lambda|<\lambda_1$, 
  then for every $|\lambda|<\lambda_1$, 
  \begin{gather*}
    \sum_{n=0}^\infty a_n\lambda^n + \sum_{n=0}^\infty b_n\lambda^n
    =\sum_{n=0}^\infty (a_n+b_n)\lambda^n,\\
    \left(\sum_{n=0}^\infty a_n\lambda^n\right)
    \left(\sum_{n=0}^\infty b_n\lambda^n\right)
    =\sum_{n=0}^\infty c_n\lambda^n,\quad
    c_n=\sum_{i=0}^na_ib_{n-i}.
  \end{gather*}
\end{theorem}

\begin{corollary}\label{p9b}
  Given 
  $\sum_{n=0}^\infty\alpha_nX^n,\sum_{n=0}^\infty\beta_nX^n\in\mathbb{C}\{X\}$
  and $f\in\mathcal{C}$, we find 
  \begin{gather*}
    \sum_{n=0}^\infty\alpha_nf^n+\sum_{n=0}^\infty\beta_nf^n
    =\sum_{n=0}^\infty(\alpha_n+\beta_n)f^n,\\
    \left(\sum_{n=0}^\infty\alpha_nf^n\right)
    \left(\sum_{n=0}^\infty\beta_nf^n\right)
    =\sum_{n=0}^\infty\gamma_n f^n,\quad
    \gamma_n=\sum_{i=0}^n\alpha_i\beta_{n-i}.
  \end{gather*}
\end{corollary}

\begin{proof}
  By Theorem \ref{p8b}, 
  $\sum_{n=0}^\infty\alpha_n\lambda^nf^n$ and 
  $\sum_{n=0}^\infty\beta_n\lambda^nf^n$
  ($\lambda\in\mathbb{C}$) converge.  
  Hence apply Theorem \ref{p9a} under $\lambda_1=2$, $a_n=\alpha_nf^n$ and
  $b_n=\beta_nf^n$.  
\end{proof}

\begin{definition}[Part III, Chapter V, \S37]\label{p11a}
  Let $\alpha\in\mathbb{C}$.  
  Define the automorphism $T^\alpha\colon Q(\mathcal{C})\to Q(\mathcal{C})$ by
  \[
    T^\alpha f=\{\e^{\alpha t}f(t)\},\quad f\in\mathcal{C}.
  \]
  For any $\beta\in\mathbb{C}$, we see $T^\alpha\beta=\beta$, 
  and for the differential operator $s$,  
  \[
    T^\alpha s=\frac{T^\alpha 1}{T^\alpha\{1\}}
    =\frac{1}{\{\e^{\alpha t}\}}=s-\alpha.  
  \]
\end{definition}

\begin{definition}[Part III, Chapter VIII, \S57]\label{p11b}
  Define the mapping $D\colon \mathcal{C}\to\mathcal{C}$ by
  \[
    Da=\{-ta(t)\}.  
  \]
  Then it satisfies 
  \begin{gather*}
    D(a+b)=Da+Db,\\
    D(ab)=D(a)b+aD(b),
  \end{gather*}
  and is uniquely extended to a derivation of $Q(\mathcal{C})$ by
  \[
    D\left(\frac{a}{b}\right)=\frac{D(a)b-aD(b)}{b^2}.
  \]
  For any $\alpha\in \mathbb{C}$,
  \[
    D\alpha=D\left(\frac{\{\alpha\}}{\{1\}}\right)
    =\frac{D(\{\alpha\})\{1\}-\{\alpha\}D(\{1\})}{\{1\}^2}
    =\frac{\{-\alpha t\}\{1\}-\{\alpha\}\{-t\}}{\{1\}^2}=0,
  \]
  and for the differential operator $s$,
  \[
    Ds=D\left(\frac{1}{\{1\}}\right)
    =\frac{D(1)\{1\}-1D(\{1\})}{\{1\}^2}
    =\frac{-\{-t\}}{\{t\}}=1. 
  \]
  In this sense, Mikusi\'nski writes $D=d/ds$.  
\end{definition}

\section{Difference fields on $Q(\mathcal{C})$}\label{s:dcfield}

Definition \ref{p11a} shows that $(Q(\mathcal{C}),T^\alpha)$ is a difference field.  
In addition, for any $a,b\in\mathcal{C}$ $(b\neq 0)$, 
\[
  \frac{d}{ds}T^\alpha a=\frac{d}{ds}\{\e^{\alpha t}a(t)\}
  =\{-t\e^{\alpha t}a(t)\}=T^\alpha\{-ta(t)\}=T^\alpha\frac{d}{ds}a,
\]
and thus
\begin{equation*}
  \begin{aligned}
  \frac{d}{ds}T^\alpha\left(\frac{a}{b}\right)
  &=\frac{d}{ds}\left(\frac{T^\alpha a}{T^\alpha b}\right)
  =\frac{\left(\dfrac{d}{ds}T^\alpha a\right)(T^\alpha b)-(T^\alpha a)\left(\dfrac{d}{ds}T^\alpha b\right)}{(T^\alpha b)^2}\\
  &=\frac{\left(T^\alpha\dfrac{d}{ds} a\right)(T^\alpha b)-(T^\alpha a)\left(T^\alpha\dfrac{d}{ds} b\right)}{(T^\alpha b)^2}\\
  &=T^\alpha\frac{d}{ds}\left(\frac{a}{b}\right).
  \end{aligned}
\end{equation*}
Therefore $(Q(\mathcal{C}),d/ds,T^\alpha)$ is 
a DT field with 
\[
  \dfrac{d}{ds}T^\alpha=T^\alpha\dfrac{d}{ds}.
\]

In this section, we shall define another transforming operator.  

\begin{definition}\label{p13a}
  For $q>0$, let $\tau_q\colon\mathcal{C}\to\mathcal{C}$ denote the bijection 
   defined by
  \[
    \tau_qa=\{qa(qt)\}.
  \]
\end{definition}

\begin{lemma}\label{p13b}
  For any $a,b\in\mathcal{C}$, 
  \begin{gather*}
    \tau_q(a+b)=\tau_qa+\tau_qb,\\
    \tau_q(ab)=\tau_q(a)\tau_q(b).
  \end{gather*}
\end{lemma}

\begin{proof}
  The first formula is seen by 
  \[
    \tau_q(a+b)=\tau_q\{a(t)+b(t)\}
    =\{q(a(qt)+b(qt))\}
    =\tau_qa+\tau_qb.
  \]
  The second formula is proved in the following way,
  \begin{equation*}
  \begin{aligned}
    \tau_q(a)\tau_q(b)&=\{qa(qt)\}\{qb(qt)\}\\
    &=\left\{\int_0^tqa(qt-q\tau)qb(q\tau)d\tau\right\}\\
    &=\left\{\int_0^{qt}qa(qt-\sigma)qb(\sigma)\frac{d\sigma}{q}\right\}\\
    &=\tau_q\left\{\int_0^ta(t-\sigma)b(\sigma)d\sigma\right\}\\
    &=\tau_q(ab).
  \end{aligned}
  \end{equation*}
\end{proof}

\begin{lemma}\label{p14a}
  The bijection $\tau_q$ is extended to $\tau\colon Q(\mathcal{C})\to Q(\mathcal{C})$ by
  \[
    \tau\left(\frac{a}{b}\right)=\frac{\tau_q a}{\tau_q b}\quad
    (a,b\in\mathcal{C},\ b\neq \{0\}).
  \]
  It is well-defined, and an automorphism of $Q(\mathcal{C})$.  
  We keep using the symbol $\tau_q$ instead of $\tau$.  
\end{lemma}

\begin{proof}
  (well-definedness) Suppose $a/b=c/d$, which means $ad=bc$.  
  Hence we obtain 
  \[
    \tau_q(a)\tau_q(d)=\tau_q(b)\tau_q(c),
  \]
  and thus
  \[
    \frac{\tau_q (a)}{\tau_q (b)}=\frac{\tau_q (c)}{\tau_q (d)}.
  \]
  
  (Extension) An element $a\in\mathcal{C}$ is actually
  $ab/b$ for any non-zero $b\in\mathcal{C}$ in $Q(\mathcal{C})$, and thus
  \[
    \tau\left(\frac{al}{l}\right)=\frac{\tau_q(a)\tau_q(l)}{\tau_q(l)}
  \]
  implies $\tau(a)=\tau_q(a)$.  
  
  (Automorphism) We find $\tau$ is homomorphic in the usual way,
  \begin{equation*}
    \begin{aligned}
    \tau\left(\frac{a}{b}+\frac{c}{d}\right)
    &=\tau\left(\frac{ad+bc}{bd}\right)
    =\frac{\tau_q(a)\tau_q(d)+\tau_q(b)\tau_q(c)}{\tau_q(b)\tau_q(d)}
    =\frac{\tau_q(a)}{\tau_q(b)}+\frac{\tau_q(c)}{\tau_q(d)}\\
    &=\tau\left(\frac{a}{b}\right)+\tau\left(\frac{c}{d}\right),
    \end{aligned}
  \end{equation*}
  \begin{equation*}
    \tau\left(\frac{a}{b}\cdot\frac{c}{d}\right)
    =\frac{\tau_q(ac)}{\tau_q(bd)}
    =\frac{\tau_q(a)}{\tau_q(b)}\cdot\frac{\tau_q(c)}{\tau_q(d)}
    =\tau\left(\frac{a}{b}\right)\tau\left(\frac{c}{d}\right).
  \end{equation*}
  Since $\tau_q\colon\mathcal{C}\to\mathcal{C}$ is surjective, $\tau$ is automorphic.  
\end{proof}

\begin{theorem}\label{p15a}
  The difference field $(Q(\mathcal{C}),\tau_q)$ has the following properties.  
  \begin{enumerate}
  \item\label{p15a1} $\tau_q\alpha=\alpha$ $(\alpha\in\mathbb{C})$.
  \item\label{p15a2} $\tau_ql=ql$.
  \item\label{p15a3} $\tau_qs=s/q$. 
  \item\label{p15a4} The derivation $D$ defined by 
    \[D=-s^2\dfrac{d}{ds}\]
    satisfies $Dl=1$, by which 
    $(Q(\mathcal{C}),D,\tau_q)$ is a DT field with 
    \[D\tau_q=q\tau_qD\quad 
    \left(\text{or}\quad lD\tau_q=\tau_q\left(lD\right)\right).\]
  \end{enumerate}
\end{theorem}

\begin{proof}
  (\ref{p15a1})
  \[
    \tau_q\alpha=\tau_q\left(\frac{\{\alpha\}}{\{1\}}\right)
    =\frac{\{q\alpha\}}{\{q\}}
    =\frac{\alpha\{q\}}{\{q\}}
    =\alpha.
  \]
  (\ref{p15a2})
  \[
    \tau_ql=\tau_q\{1\}=\{q\}=q\{1\}=ql.
  \]
  (\ref{p15a3})
  \[
    \tau_qs=\tau_q\left(\frac{1}{l}\right)
    =\frac{1}{ql}=\frac{s}{q}.
  \]
  (\ref{p15a4})
  \[
    Dl=-s^2\frac{d}{ds}s^{-1}=-s^2(-s^{-2})=1.
  \]
  For any $a\in\mathcal{C}$, 
  \[
    \tau_qDa=\tau_q(-s^2\{-ta(t)\})
    =-\frac{s^2}{q^2}\{-q^2ta(qt)\}
    =s^2\{ta(qt)\},
  \]
  and thus
  \[
    D\tau_qa=D\{qa(qt)\}
    =-s^2\{-qta(qt)\}=qs^2\{ta(qt)\}=q\tau_qDa.
  \]
  Hence for any $a,b\in\mathcal{C}$ $(b\neq 0)$,
  \begin{equation*}
    \begin{aligned}
    D\tau_q\left(\frac{a}{b}\right)
    &=D\frac{\tau_q(a)}{\tau_q(b)}
    =\frac{\left(D\tau_q(a)\right)\tau_q(b)-\tau_q(a)\left(D\tau_q(b)\right)}{\tau_q(b)^2}\\
    &=\frac{\left(q\tau_qDa\right)\tau_q(b)-\tau_q(a)\left(q\tau_qDb\right)}{\tau_q(b)^2}\\
    &=q\tau_qD\left(\frac{a}{b}\right).
    \end{aligned}
  \end{equation*}
\end{proof}

Therefore $(Q(\mathcal{C}),\tau_q)$ is a difference overfield of 
the field of rational functions $\mathbb{C}(l)$ as a $q$-difference field.  
Moreover, $Q(\mathcal{C})/\mathbb{C}(h)$ becomes a difference field extension
of Mahler type when $q$ is a unit fraction.  

\begin{theorem}\label{p16a}
  Suppose $q$ is a unit fraction, that is  $q=1/d$, $d\in\mathbb{Z}_{>1}$.  
  We use the symbol $\sigma_d=\tau_q$ for clarity.  
  Then it satisfies $\sigma_dh=h^d$.  
  Define the derivation $D'$ by
  \[D'=-h^{-1}\dfrac{d}{ds},\] 
  so that $D'h=1$.  
  Then $(Q(\mathcal{C}),D',\sigma_d)$ is a DT field with
  \[D'\sigma_d=dh^{d-1}\sigma_dD'.\]
\end{theorem}

\begin{proof}
  We have seen in the proof of Theorem \ref{p7a} that
  for any $\lambda\geq 0$,
  \[
    h^\lambda=\frac{1}{l^2}\{\max\{0,t-\lambda\}\}.
  \]
  Hence $\sigma_dh$ is calculated as follows, 
  \[
    \sigma_dh=\tau_qh=\frac{1}{q^2l^2}\{q\max\{0,qt-1\}\}
    =\frac{1}{l^2}\{\max\{0,t-\dfrac{1}{q}\}\}
    =h^d.
  \]
  To prove $D'h=1$, we start with 
  \begin{equation*}
    \begin{aligned}
    \frac{d}{ds}h&=\frac{d}{ds}(s^2\{\max\{0,t-1\}\})\\
    &=2s\{\max\{0,t-1\}\}+s^2\{-t\max\{0,t-1\}\}\\
    &=2s^2\left\{\int_0^t\max\{0,\tau-1\}d\tau\right\}
    -s^2\{t\max\{0,t-1\}\}.\\
    \end{aligned}
  \end{equation*}
  By
  \begin{equation*}
    \begin{aligned}
    \int_0^t\max\{0,\tau-1\}d\tau
    &=\begin{cases}
      0 & (0\leq t\leq 1),\\
      \displaystyle \int_1^t\tau-1d\tau & (1<t<\infty) 
    \end{cases}\\
    &=\begin{cases}
      0 & (0\leq t\leq 1),\\
      \dfrac{t^2}{2}-t+\dfrac{1}{2} & (1<t<\infty),
    \end{cases}
    \end{aligned}
  \end{equation*}
  we see
  \begin{equation*}
    \frac{d}{ds}h
    =s^2\left\{
      \begin{aligned}
      & 0 && (0\leq t\leq 1),\\
      & -t+1 && (1<t<\infty)
      \end{aligned}
    \right\}
    =-h.
  \end{equation*}
  Hence by definition,
  \[
    D'h=-\frac{1}{h}\frac{d}{ds}h=1.
  \]
  Lastly, for any $a\in Q(\mathcal{C})$, 
  \begin{equation*}
    \begin{aligned}
    \sigma_dD'a
    &=\sigma_d\left(-\frac{1}{h}\frac{d}{ds}a\right)
    =-\frac{1}{h^d}\tau_q\frac{d}{ds}a
    =-\frac{1}{h^d}\tau_q\left(-l^2Da\right)\\
    &=-\frac{1}{h^d}(-q^2l^2)\left(\frac{1}{q}D\tau_qa\right)
    =\frac{ql^2}{h^d}\left(-s^2\frac{d}{ds}\sigma_da\right)\\
    &=-\frac{q}{h^d}\left(-hD'\sigma_da\right)
    =\frac{q}{h^{d-1}}D'\sigma_da,
    \end{aligned}
  \end{equation*}
  where $D=-s^2(d/ds)$ is the same one as in Theorem \ref{p15a}.  
  Therefore we obtain
  \[
    D'\sigma_da=dh^{d-1}\sigma_dD'a.
  \]
\end{proof}

\begin{remark}
  Using Theorem \ref{th_deg} under $\mathcal{K}=(\mathbb{C},\textrm{id})$
  and $\mathcal{L}=(Q(\mathcal{C}),\sigma_2)$, 
  we can prove algebraic independence of $s$ and $h^\lambda$ $(\lambda>0)$
  over $\mathbb{C}$.  
  Let $\lambda>0$.  Then we find that $h^\lambda$ satisfies 
  $\sigma_2h^\lambda=(h^\lambda)^2$ by 
  \[
    \sigma_2h^\lambda=\frac{4}{l^2}\{2^{-1}\max\{0,2^{-1}t-\lambda\}\}
    =\frac{1}{l^2}\{\max\{0,t-2\lambda\}\}
    =h^{2\lambda}=(h^\lambda)^2.
  \]
  It is a first-order rational difference equation of degree 2.  
  On the other hand, we have $\sigma_2s=2s$, which is one of degree 1.  
  Set $c_1=1$, $c_2=2$, $f_{11}=s$ and $f_{21}=h^\lambda$ in Theorem \ref{th_deg}.
  Since each of $s$ and $h^\lambda$ is transcendental over $\mathbb{C}$
  (see Example \ref{p2a} and Definition \ref{p5a}), 
  they are algebraically independent over $\mathbb{C}$.  
  In other words, the translation operator $h^\lambda$ $(\lambda>0)$ 
  cannot be expressed
  algebraically by $s$ unlike the exponential function or Bessel function.  
\end{remark}

\section{The $q$-difference subring $\mathbb{C}\{l\}$}\label{s:qdc}

\begin{definition}\label{p18a}
  In this section, we shall consider the set
  \[
    \mathbb{C}\{l\}
    =\left\{\sum_{n=0}^\infty\alpha_nl^n\ \middle|\ \sum_{n=0}^\infty\alpha_nX^n\in\mathbb{C}\{X\}\right\}\subset Q(\mathcal{C}).
  \]
  It is well-defined by Theorem \ref{p8b}, 
  and a subring of $Q(\mathcal{C})$ by Corollary \ref{p9b}.  
\end{definition}

\begin{lemma}\label{p18b}
  The expression $\sum_{n=0}^\infty\alpha_nl^n\in\mathbb{C}\{l\}$ 
  $(\sum_{n=0}^\infty \alpha_n X^n\in\mathbb{C}\{X\})$ is unique, 
  and thus $\mathbb{C}\{l\}\cong\mathbb{C}\{X\}\subset\mathbb{C}[[X]]$
  as rings.  
  Since $l$ is transcendental over $\mathbb{C}$, 
  $\mathbb{C}[[l]]$ denotes the ring of formal power series. 
  Therefore, we may think $\mathbb{C}\{l\}\hookrightarrow\mathbb{C}[[l]]$
  as rings in the natural way.  
\end{lemma}

\begin{proof}
  Let $\sum_{n=0}^\infty\alpha_n X^n\in\mathbb{C}\{X\}$ and suppose 
  $
    \sum_{n=0}^\infty\alpha_nl^n=0. 
  $
  Multiplying both sides by $l$, we obtain the following by Example \ref{p8c}, 
  \[
    \{0\}=l\sum_{n=0}^\infty\alpha_nl^n=\sum_{n=0}^\infty\alpha_nl^{n+1}
    =\sum_{n=1}^\infty\alpha_{n-1}l^n
    =\left\{\sum_{n=1}^\infty\frac{\alpha_{n-1}t^{n-1}}{(n-1)!}\right\}.
  \]
  Since it implies
  \[
    \sum_{n=1}^\infty\frac{\alpha_{n-1}t^{n-1}}{(n-1)!}=0\quad
    (t\in[0,\infty)),
  \]
  we see $\alpha_0=\alpha_1=\alpha_2=\dots=0$.  
\end{proof}

\begin{theorem}\label{p19a}
  Let $\sum_{n=0}^\infty\alpha_nX^n\in\mathbb{C}\{X\}$.  Then
  \begin{align*}
    \tau_q\left(\sum_{n=0}^\infty\alpha_nl^n\right)
    =&\sum_{n=0}^\infty\alpha_nq^nl^n,\\
    D\left(\sum_{n=0}^\infty\alpha_nl^n\right)
    =&\sum_{n=1}^\infty n\alpha_nl^{n-1},\quad D=-s^2(d/ds).
  \end{align*}
  Hence $\mathbb{C}\{l\}$ is a DT subring of $(Q(\mathcal{C}),D,\tau_q)$,
  and its quotient field $Q(\mathbb{C}\{l\})$ is a DT subfield.  
  Moreover, we may think $Q(\mathbb{C}\{l\})$ is a DT subfield of 
  the field $\mathbb{C}((l))$ of formal power series. 
\end{theorem}

\begin{proof} Each of the equations is seen as follows,
  \begin{equation*}
    \begin{aligned}
    \tau_q\left(\sum_{n=0}^\infty\alpha_nl^n\right)
    &=\tau_q\left(\alpha_0+\sum_{n=1}^\infty\alpha_nl^n\right)
    =\alpha_0+\tau_q\left\{\sum_{n=1}^\infty\frac{\alpha_nt^{n-1}}{(n-1)!}\right\}\\
    &=\alpha_0+\left\{q\sum_{n=1}^\infty\frac{\alpha_nq^{n-1}t^{n-1}}{(n-1)!}\right\}\\
    &=\alpha_0+\sum_{n=1}^\infty\alpha_nq^nl^n
    =\sum_{n=0}^\infty\alpha_nq^nl^n,
    \end{aligned}
  \end{equation*}
  \begin{equation*}
    \begin{aligned}
    D\left(\sum_{n=0}^\infty\alpha_nl^n\right)
    &=-s^2\frac{d}{ds}\left(\alpha_0+\sum_{n=1}^\infty\alpha_nl^n\right)
    =-s^2\frac{d}{ds}\left\{\sum_{n=1}^\infty\frac{\alpha_nt^{n-1}}{(n-1)!}\right\}\\
    &=-s^2\left\{-t\sum_{n=1}^\infty\frac{\alpha_nt^{n-1}}{(n-1)!}\right\}
    =s^2\left\{\sum_{n=1}^\infty\frac{\alpha_nt^n}{(n-1)!}\right\}\\
    &=s^2\left\{\sum_{n=2}^\infty\frac{\alpha_{n-1}t^{n-1}}{(n-2)!}\right\}
    =s^2\left\{\sum_{n=2}^\infty\frac{(n-1)\alpha_{n-1}t^{n-1}}{(n-1)!}\right\}\\
    &=s^2\sum_{n=2}^\infty(n-1)\alpha_{n-1}l^n
    =s^2\sum_{n=1}^\infty n\alpha_nl^{n+1}
    =\sum_{n=1}^\infty n\alpha_nl^{n-1}.
    \end{aligned}
  \end{equation*}
\end{proof}

By the way, $T^\alpha$ defined in Definition \ref{p11a} satisfies
\[
  T^\alpha l=T^\alpha\left(\frac{1}{s}\right)=\frac{1}{s-\alpha}
  =\frac{l}{1-\alpha l}.  
\]
We will see that it makes $\mathbb{C}\{l\}$ a difference subring of 
$(Q(\mathcal{C}),T^\alpha)$.

\begin{theorem}\label{p21a}
  Let $\sum_{n=0}^\infty\alpha_nX^n\in\mathbb{C}\{X\}$.  Then
  \[
    T^\alpha\left(\sum_{n=0}^\infty\alpha_nl^n\right)
    =\sum_{n=0}\beta_nl^n,
  \]
  where $\beta_n$ is determined by 
  \[
    \sum_{n=0}^\infty\beta_nZ^n
    =\sum_{n=0}^\infty\alpha_n\left(\frac{Z}{1-\alpha Z}\right)^n
    \in\mathbb{C}\{Z\}.
  \]
  The derivative is 
  \[
    \frac{d}{ds}\left(\sum_{n=0}^\infty\alpha_nl^n\right)
    =\sum_{n=1}^\infty -n\alpha_nl^{n+1}.
  \]
  Therefore $\mathbb{C}\{l\}$ is a DT subring of $(Q(\mathcal{C}),d/ds,T^\alpha)$, and its quotient field $Q(\mathbb{C}\{l\})$
  is a DT subfield.  
  Moreover, we may think $Q(\mathbb{C}\{l\})$ is a DT subfield of
  $\mathbb{C}((l))$.  
\end{theorem}

\begin{proof}
  Firstly, we shall check
  \[
    \beta_{n+1}=\sum_{i=0}^n\binom{n}{i}\alpha^{n-i}\alpha_{i+1}.
  \]
  In the definition of $\beta_n$, the right side is 
  \begin{equation*}
    \begin{aligned}
    \sum_{n=0}^\infty\alpha_n\left(\frac{Z}{1-\alpha Z}\right)^n
    &=\sum_{n=0}^\infty\alpha_nZ^n(1+\alpha Z+\alpha^2Z^2+\cdots)^n\\
    &=\alpha_0+\sum_{n=1}^\infty\alpha_nZ^n\left(\sum_{i=0}^\infty\binom{i+n-1}{n-1}\alpha^iZ^i\right)\\
    &=\alpha_0+\sum_{n=1}^\infty\left(\sum_{j=1}^n\alpha_j\binom{n-1}{j-1}\alpha^{n-j}\right)Z^n.
    \end{aligned}
  \end{equation*}
  Hence 
  \[
    \beta_{n+1}=\sum_{j=1}^{n+1}\alpha_j\binom{n}{j-1}\alpha^{n+1-j}
    =\sum_{i=0}^n\alpha_{i+1}\binom{n}{i}\alpha^{n-i}.
  \]
  Sencondly, by the above equality, 
  \begin{equation*}
    \begin{aligned}
    T^\alpha\left(\sum_{n=0}^\infty\alpha_nl^n\right)
    &=\alpha_0+T^\alpha\left\{\sum_{n=1}^\infty\frac{\alpha_nt^{n-1}}{(n-1)!}\right\}
    =\alpha_0+\left\{\e^{\alpha t}\sum_{n=1}^\infty\frac{\alpha_nt^{n-1}}{(n-1)!}\right\}\\
    &=\alpha_0+\left\{
    \left(\sum_{n=0}^\infty\frac{\alpha^n}{n!}t^n\right)
    \left(\sum_{n=0}^\infty\frac{\alpha_{n+1}}{n!}t^n\right)\right\}\\
    &=\alpha_0+\left\{
    \sum_{n=0}^\infty\left(\sum_{i=0}^n\frac{\alpha^{n-i}}{(n-i)!}\frac{\alpha_{i+1}}{i!}\right)t^n\right\}\\
    &=\alpha_0+\left\{
    \sum_{n=0}^\infty\left(\sum_{i=0}^n\binom{n}{i}\alpha^{n-i}\alpha_{i+1}\right)\frac{t^n}{n!}\right\}\\
    &=\alpha_0+\left\{\sum_{n=0}^\infty\beta_{n+1}\frac{t^n}{n!}\right\}
    =\alpha_0+\left\{\sum_{n=1}^\infty\frac{\beta_nt^{n-1}}{(n-1)!}\right\}\\
    &=\sum_{n=0}^\infty\beta_nl^n.
    \end{aligned}
  \end{equation*}
  Lastly, the derivative is caluculated as follows,
  \[
    \frac{d}{ds}\left(\sum_{n=0}^\infty\alpha_nl^n\right)
    =-l^2D\left(\sum_{n=0}^\infty\alpha_nl^n\right)
    =-l^2\sum_{n=1}^\infty n\alpha_nl^{n-1}
    =\sum_{n=1}^\infty -n\alpha_nl^{n+1}.
  \]
\end{proof}

\section{The difference subring $\mathbb{C}[[h]]$ of Mahler type}\label{s:Mahler}

\begin{definition}\label{p23a}
  In this section, we shall consider the set
  \[
    \mathbb{C}[[h]]
    =\left\{\sum_{n=0}^\infty\alpha_nh^n\ \middle|\ \alpha_n\in\mathbb{C}\right\}
    \subset Q(\mathcal{C}),
  \]
  which is well-defined by Theorem \ref{p7a}.  
\end{definition}

\begin{lemma}\label{p23b}
  The set $\mathbb{C}[[h]]$ is a subring of $Q(\mathcal{C})$
  and satisfies the following, 
  \begin{gather*}
    \sum_{n=0}^\infty\alpha_nh^n+\sum_{n=0}^\infty\beta_nh^n
    =\sum_{n=0}^\infty(\alpha_n+\beta_n)h^n,\\
    \left(\sum_{n=0}^\infty\alpha_nh^n\right)
    \left(\sum_{n=0}^\infty\beta_nh^n\right)
    =\sum_{n=0}^\infty\left(\sum_{i=0}^n\alpha_i\beta_{n-i}\right)h^n.
  \end{gather*}
\end{lemma}

\begin{proof}
  Since $\sum_{n=0}^\infty\alpha_nh^n\lambda^n$ and $\sum_{n=0}^\infty\beta_nh^n\lambda^n$
  converge for any $\lambda\in\mathbb{C}$, 
  the above equations follow from Theorem \ref{p9a}.  
  Hence $\mathbb{C}[[h]]$ is a subring of $Q(\mathcal{C})$.
\end{proof}

\begin{lemma}[\cite{Mikusinski1983}, Part II, Chapter II, \S5]\label{p24a}
  The expression $\sum_{n=0}^\infty\alpha_nh^n\in\mathbb{C}[[h]]$ is unique,
  and thus $\mathbb{C}[[h]]\cong\mathbb{C}[[X]]$ as rings.  
\end{lemma}

\begin{proof}
  Suppose $\sum_{n=0}^\infty\alpha_nh^n=0$.  
  By Theorem \ref{p7a}, it means 
  \[
    \left\{\sum_{n=0}^\infty\alpha_n\max\{0,t-n\}\right\}=\{0\}.
  \]
  If $\alpha_m\neq0$ for some $m$, then for the minimum $m$, 
  the value of the above equation at $t=m+1$ would be 
  \[
    \alpha_m\cdot 1+\alpha_{m+1}\cdot 0+\alpha_{m+2}\cdot 0+\dots=0, 
  \]
  which is a contradiction.  
  Hence we conclude $\alpha_0=\alpha_1=\dots=0$.  
\end{proof}

\begin{theorem}\label{p25a}
  Suppose $q=1/d$, $d\in\mathbb{Z}_{>1}$.  Let $\sigma_d=\tau_q$ and 
  $\sum_{n=0}^\infty\alpha_nh^n\in\mathbb{C}[[h]]$.
  Then 
  \begin{align*}
    \sigma_d\left(\sum_{n=0}^\infty\alpha_nh^n\right)
    =&\sum_{n=0}^\infty\alpha_nh^{dn},\\
    D'\left(\sum_{n=0}^\infty\alpha_nh^n\right)
    =&\sum_{n=1}^\infty n\alpha_nh^{n-1}, \quad D'=-h^{-1}(d/ds)
  \end{align*}
  Hence $\mathbb{C}[[h]]$ is a DT subring of 
  $(Q(\mathcal{C}),D',\sigma_d)$ and 
  the quotient field $\mathbb{C}((h))$ is a DT subfield.
\end{theorem}

\begin{proof} The first equation is obtained as follows, 
  \begin{equation*}
    \begin{aligned}
    \sigma_d\left(\sum_{n=0}^\infty\alpha_nh^n\right)
    &=\tau_q\left(\frac{1}{l^2}\left\{\sum_{n=0}^\infty\alpha_n\max\{0,t-n\}\right\}\right)\\
    &=\frac{1}{q^2l^2}\left\{q\sum_{n=0}^\infty\alpha_n\max\{0,qt-n\}\right\}\\
    &=\frac{1}{l^2}\left\{\sum_{n=0}^\infty\alpha_n\max\{0,t-dn\}\right\}\\
    &=\sum_{n=0}^\infty\alpha_nh^{dn}.
    \end{aligned}
  \end{equation*}
  
  To prove the second equation, we start with the derivative by $d/ds$, 
  \begin{equation*}
    \begin{aligned}
    \frac{d}{ds}\left(\sum_{n=0}^\infty\alpha_nh^n\right)
    &=\frac{d}{ds}\left(s^2\left\{\sum_{n=0}^\infty\alpha_n\max\{0,t-n\}\right\}\right)\\
    &=2s\left\{\sum_{n=0}^\infty\alpha_n\max\{0,t-n\}\right\}
    +s^2\left\{-t\sum_{n=0}^\infty\alpha_n\max\{0,t-n\}\right\}\\
    &=2s^2\left\{\sum_{n=0}^\infty\alpha_n\int_0^t\max\{0,\tau-n\}d\tau\right\}\\
    &\qquad -s^2\left\{\sum_{n=0}^\infty\alpha_nt\max\{0,t-n\}\right\}\\
    &=s^2\left\{\sum_{n=0}^\infty\alpha_n\left(2\int_0^t\max\{0,\tau-n\}d\tau-t\max\{0,t-n\}\right)\right\}.
    \end{aligned}
  \end{equation*}
  The integral is calculated as follows, 
  \begin{equation*}
    \begin{aligned}
    \int_0^t\max\{0,\tau-n\}d\tau
    &=\begin{cases}
      0&(0\leq t\leq n),\\
      \displaystyle \int_n^t\tau-n\ d\tau &(n<t<\infty)
    \end{cases}\\
    &=\begin{cases}
      0&(0\leq t\leq n),\\
      \displaystyle \frac{t^2}{2}-nt+\frac{n^2}{2} &(n<t<\infty),
    \end{cases}\\
    \end{aligned}
  \end{equation*}
  \begin{equation*}
    \begin{aligned}
    2\int_0^t\max\{0,\tau-n\}d\tau-t\max\{0,t-n\}
    &=\begin{cases}
      0 & (0\leq t\leq n),\\
      -nt+n^2 & (n<t<\infty)
    \end{cases}\\
    &=-n\begin{cases}
      0 & (0\leq t\leq n),\\
      t-n & (n<t<\infty)
    \end{cases}\\
    &=-n\max\{0,t-n\}.
    \end{aligned}
  \end{equation*}
  Hence we find 
  \begin{equation*}
    \begin{aligned}
    \frac{d}{ds}\left(\sum_{n=0}^\infty\alpha_nh^n\right)
    &=s^2\left\{\sum_{n=0}^\infty -n\alpha_n\max\{0,t-n\}\right\}\\
    &=\sum_{n=1}^\infty -n\alpha_nh^n,
    \end{aligned}
  \end{equation*}
  and thus by the definition of $D'$,
  \[
    D'\left(\sum_{n=0}^\infty\alpha_nh^n\right)
    =-\frac{1}{h}\sum_{n=1}^\infty-n\alpha_nh^n
    =\sum_{n=1}^\infty n\alpha_nh^{n-1}. 
  \]
\end{proof}

\begin{acknowledgment}
  This work was partially supported by Japan Society for the Promotion of Science (JSPS) KAKENHI Grant Number 23K03144.
\end{acknowledgment}

\end{document}